\documentclass[a4paper,11pt]{amsart}
\usepackage{tikz}
\usepackage{graphicx}
\usepackage{psfrag}
\usepackage{mathrsfs}
\usepackage{color}
\usepackage{amsmath,amssymb, amsfonts}

\usepackage{hyperref}
\usepackage[active]{srcltx}
\usepackage[normalem]{ulem}

\numberwithin{equation}{section}

\usepackage{thmtools}
\usepackage{thm-restate}

\newtheorem{maintheorem}{Theorem}

\newtheorem{maincorollary}[maintheorem]{Corollary}

\newtheorem{theorem}{Theorem}[section]

\newtheorem{proposition}[theorem]{Proposition}
\newtheorem{lemma}[theorem]{Lemma}
\newtheorem{definition}[theorem]{Definition}

\newtheorem{example}{Example}
\newtheorem{remark}[theorem]{Remark}
\newtheorem{problem}{Problem}

\title{Generating positive geometric entropy from recurrent leaves}
\author{Gabriel Ponce}
\address{Departamento de Matem\'atica,
  IMECC-UNICAMP Campinas-SP, Brazil.}
  \email{gaponce@ime.unicamp.br}
\thanks{The author had the financial support of FAPESP process \# 2016/05384-0 and FAEPEX process \#2334/16}
\date{}                                         

\begin{document}
\maketitle

\begin{abstract}

In this paper we introduce a new $C^r-$perturbation procedure, with respect to the $C^r$-Epstein topology, for $C^r$-foliations by surfaces. Using this perturbation procedure we show how one can use the existence of recurrent leaves of certain $C^r-$foliation $\mathcal F$ to obtain a foliation $\mathcal G$, $C^r-$close to $\mathcal F$ in the $C^r-$Epstein topology, which has a resilient leaf. In particular, one can take advantage of recurrence property to construct examples of $C^r-$foliations by surfaces, $C^r-$close to each other and such that one of them has a resilient leaf while the other is Riemannian (therefore has trivial dynamics).
\end{abstract}

\section{Introduction}

Inspired by Smale's work \cite{Smale} on the dynamics of $C^r$ diffeomorphisms on closed Riemannian manifolds, the study of the dynamics of foliations aims to understand the behavior of a foliation in terms of its dynamical characteristics such as the existence of minimal sets and orbit growth of leaves for example. As in the case of diffeomorphisms concepts such as transitivity, hyperbolicity and entropy are in the core of the theory though each of these concepts are treated in a completely different way when working with foliations.

Topological entropy of a map or flow is a dynamical invariant which captures the exponential complexity of the given dynamical system.The analogous notion of topological entropy for the foliation theory context was introduced by Ghys-Langevin-Walczak \cite{GLW} in a similar way to the definition of topological entropy of maps given by Bowen \cite{Bowen71}. Though the definition resembles the definition of topological entropy for diffeomorphisms, dealing with foliations is much more complicate since we must deal with a pseudogroup of transformations (the holonomy pseudogroup) instead of iterates of a unique map. Nevertheless, the concept of geometric entropy has provided very interesting results which helps us to obtain dynamical informations of the foliation such as the existence or not of probability invariant transverse measures for foliations with positive entropy, existence of resilient leaves etc. For diffeomorphisms there are several results relating transitivity of the map to the existence of maps with positive entropy arbitrarily $C^1-$close to it. Such results often uses $C^1-$perturbation methods to construct maps with homoclinic intersections obtaining, as a consequence of \cite{BurnsWeiss}, a map with positive entropy $C^1-$close to the original system . The solution of Palis's conjecture by S. Crovisier \cite{Crovisier} is an outstanding result which shows that every diffeomorphism which is not in the closure of the set of Morse-Smale diffeomorphisms can be approximated by diffeomorphisms with homoclinic intersections and, consequently, the set of $C^1-$diffeomorphisms which have stably zero topological entropy is exactly the closure of the set of Morse-Smale diffeomorphisms. In the light of the results for the discrete case it is natural to ask wether we can take advantage of recurrence properties of foliations to approximate certain foliations by foliations with positive geometric entropy.

The main obstacle to obtain $C^r-$perturbation results for diffeomorphisms is the required control on the derivatives. For foliations we will see that the main obstacle is not intrinsically related to the $C^r-$Epstein topology only but to the topology of the leaves itself. While on the diffeomorphisms case we can choose an open ball around a point and perturb the function inside this ball, for the case of foliations this procedure is not so easy for after the perturbation of the leaves restricted to a local chart there is no clear way to glue this perturbed part with the non-perturbed one preserving the regularity. Here is where the concept of splitting charts plays a central role. Roughly speaking we say that a foliation admits a family splitting charts (Definition ~\ref{definition:jointly.separating.charts}) if one can choose a family of local charts which cuts all the leaves which cross the domains of all the local charts. On splitting charts it is possible to cut the foliation, perturb it and then paste it back to the non-perturbed part. 

Our main Theorem shows how to use recurrence to produce resilient leaves which is the analogue of homoclinic intersection for foliation theory.

\begin{maintheorem}\label{theorem:main} 
Let $L$ be a recurrent leaf of a codimension$-1$ $C^r-$foliation $\mathcal F$ of a three dimensional Riemannian $C^r-$manifold $M$.  Assume that there exists a homotopically non-trivial loop $\gamma:[0,1] \rightarrow L$ on $L$ and a family of jointly splitting charts $\Phi$ such that 
\begin{itemize}
\item[i)] for any $\varepsilon>0$, there are $0\leq s_1<s_2 \leq 1$ such that $\gamma(s_1)$ and $\gamma(s_2)$ belong to distinct connected components of $O_{\varepsilon}(\overline{Z}) \setminus \overline{Z}$ and \[\gamma([s_1,s_2]) \subset O_{\varepsilon}(\overline{Z}) \setminus \overline{Z}\] where $Z:=\bigcup_{\phi \in \Phi} \operatorname{Dom}(\phi) \cap L $ and $O_{\varepsilon}(\overline{Z})$ denotes the $\varepsilon-$neighborhood of $\overline{Z}$ in $L$;
\item[ii)] \[ \operatorname{Dom}(\phi)\cap \gamma \ne \emptyset, \text{ for some } \phi \in \Phi.\]
\end{itemize}

Then, for any neighborhood $\mathcal V$ of $\mathcal F$ in $\operatorname{Fol}^r_1(M)$, endowed with the $C^r-$Epstein topology, there exists a $C^r-$foliation $\mathcal G \in \mathcal V$ which have a resilient leaf. Consequently, if $M$ is compact $\mathcal F$ can be $C^r-$approximated by foliations with positive geometric entropy.
\end{maintheorem}


In \cite[Problem $9.11$]{HurderSurvey})  S. Hurder adressed the question of whether the positive geometric entropy is a generic property for $C^1-$foliations. Theorem \ref{theorem:main} provides a first progress in this question by giving some explicit conditions in which a $C^r-$foliation can be $C^r-$approximated by foliations with positive geometric entropy.

\begin{maincorollary} \label{coro}
There exists a smooth codimension-$1$ Riemannian foliation $\mathcal F$ on a compact manifold $M$ which can be $C^r-$approximated by $C^r-$foliations which have resilient leaves. In particular, there are smooth foliations with zero geometric entropy which can be $C^r-$approximated by foliations with positive geometric entropy.
\end{maincorollary}

Theorem ~\ref{theorem:main} is an application of a subtle $C^r-$perturbation procedure which we introduce in Section ~\ref{section:cutandglue}. The $C^r-$perturbation procedure is interesting in its own right and, due to its applicability, is actually more important then the conclusion of Theorem ~\ref{theorem:main} itself. Further applications of this $C^r-$perturbation procedure to the study of foliation dynamics will appear in a fortiori work of the author.

\subsection{Structure of the paper}
In Section ~\ref{sec:preliminaries} we recall some concepts of foliations theory, which are necessary to the development of our results, such as the $C^r-$Epstein topology, orbit of a point, recurrent leaves and resilient leaves. In section ~\ref{sec:separating} we introduce the concept of family of splitting charts, water slide functions and present some examples. In Section ~\ref{section:cutandglue} we construct a $C^r-$perturbation technic using local perturbations (by water slide functions) on a chart belonging to a family of splitting charts. Finally in Section ~\ref{sec:creatingresilient} we use the $C^r-$perturbation procedure of Section ~\ref{section:cutandglue} to prove Theorem ~\ref{theorem:main} and Corollary ~\ref{coro}

\section{Basic definitions on foliation dynamics} \label{sec:preliminaries}
In this section we recall some definitions of foliation theory which are strictly necessary to prove our results such as the $C^r$-Epstein topology, recurrence and resilence. For an introductory approach to foliations theory we refer the reader to \cite{CandelConlonI, Camacho} and to a more general overview of foliation dynamics we refer the reader to \cite{Hurder2009}.

Let $M$ be an $m$-manifold possibly with boundary and $\mathcal F$ be a $C^r$-foliation on $M$ (see \cite[Definition $1.1.18$] {CandelConlonI}). It is well known that the $C^r$ foliated atlas $\mathcal A = \{U_{\alpha}, \varphi_{\alpha}\}_{\alpha \in \mathfrak A}$ of foliated charts defining $\mathcal F$ can be taken to be regular (see \cite[Definition $1.2.11$]{CandelConlonI}) and such that $\overline{\varphi_{\alpha}(U_{\alpha})} = [0,1]^m \subset \mathbb R^m$. Therefore along the paper we will always take foliated atlas to be of this type.
Given a regular foliated atlas $\mathcal A$ we will denote by $\Gamma_{\mathcal A}$ the holonomy pseudogroup associated to $\mathcal A$ (see \cite[Section $2.2$]{CandelConlonI}). 
Let $L$ be a leaf of a $C^r$ foliation $\mathcal F$ defined by a regular foliated atlas $\mathcal A$ and denote by $\mathcal T$ a complete transversal of $\mathcal F$ (see \cite[Section $2.2$]{CandelConlonI}). Given a curve $\gamma:[0,1] \rightarrow L$ let $\mathcal P = \{P_0,\ldots,P_k\}$ be a plaque chain (i.e, each $P_i$ is a plaque in a chart of $\mathcal A$ and $P_i \cap P_{i+1}\ne \emptyset$, $0\leq i \leq k-1$) which covers $\gamma$ in the following sense: there exists a subdivision
\[0=t_0 < t_1 < \ldots < t_k=1\]
such that $\gamma([t_i,t_{i+1}]) \subset P_i$ for $0\leq i \leq k-1$. The holonomy $h_{\mathcal P} : D \rightarrow \mathcal T$ defined on an open set $D \subset \mathcal T$ will also be referred to as the holonomy along the curve $\gamma$ and will be denoted, further in this paper, by $h_{\gamma}$.

\subsection{The $C^r$-Epstein topology}
We denote by $\operatorname{Fol}^q_r(M)$ the space of codimension-$q$ $C^r$-foliations on a manifold $M$. In the space $\operatorname{Fol}^q_r(M)$, Epstein \cite{Epstein} defined a natural topology usually referred to as the $C^r$-Epstein topology. The definition of the $C^r$-Epstein topology takes into account two axioms which should be satisfied in order to say that the topology is ``good enough'' or ``natural''. We will not recall the definition the $C^r$-Epstein topology here but we will make use the following proposition from Epstein's paper which establishes a way to say when two foliations are $C^r$-close to each other on the $C^r$-Epstein topology by comparing the holonomy maps on a locally finite family of compact sets.
 
\begin{definition}
A subset $P\subset \mathbb R^m$ is called a rectangular neighborhood if it is a product of (open, closed or semi-open) intervals.
\end{definition}

Let $\mathcal G$ be a codimension-$q$ $C^r$-foliation of $M$ given by a $C^r$ foliated atlas $\mathcal C=\{\eta_i: V_i \rightarrow \mathbb R^m\}_{i\in I}$. The maps $g_i:U_i \rightarrow \mathbb R^q$ defined by 
\[g_i:= \rho_2 \circ \eta_i, \quad i\in I ,\]
where $\rho_2:\mathbb R^{m-q} \times \mathbb R^q \rightarrow \mathbb R^{m-q} \times \mathbb R^q $ denotes the projection on the last $q$ coordinates, are called transversal projections of $\mathcal G$ (with respect to $\mathcal C$).
 
\begin{proposition}[Proposition $3.8$ of \cite{Epstein}] \label{proposition:epstein}
Let $\mathcal F$ be a codimension-$q$ $C^r$-foliation of a Riemannian manifold $M$. Let $I$ be an indexing set and let $\mathcal A=\{\phi_i:U_i \rightarrow \mathbb R^n\}_{i\in I}$ be a $C^r$ foliated atlas for $\mathcal F$. For each $i\in I$, let $K_i \subset U_i$ be a compact set and suppose that:
\begin{itemize}
\item[1)] for each $i$ the set $\phi_i(K_i)$ is a rectangular neighborhood on $\mathbb R^n$;
\item[2)] $\{K_i\}$ is a locally finite family of sets;
\item[3)] $M = \bigcup_{i\in I}\operatorname{Int}(K_i)$.
\end{itemize}
Then we obtain a neighborhood of $\mathcal F$ in the $C^r$-Epstein topology by specifying a family $\{\delta_i\}_{i\in I}$ of positive numbers in the following way: we take all $C^r$-foliations $\mathcal G$ of $M$ which have transversal projections $g_i: \operatorname{Int}(K_i) \rightarrow \mathbb R^q$, such that for $|\alpha|\leq r$,
\[|D^{(\alpha)}(g_i\circ \phi_i^{-1}-\rho_2)| < \delta_i, \quad \text{ on } \quad  \phi_i(\operatorname{Int}(K_i)), \]
where $\alpha$ denotes the multi-indexes $\alpha=(\alpha_1,...,\alpha_n)$ with $\alpha_i\geq 0$ for all $1\leq i \leq n$ with $|\alpha|:= \alpha_1+...+\alpha_n$.
\end{proposition}

From now on we will always consider the space $\operatorname{Fol}^q_r(M)$ to be endowed with the $C^r$-Epstein topology.

\subsection{Orbits, recurrence and resilient leaves}

To study the dynamics of a function is to study the behavior of orbits of points under the action of the given function. For foliations the definition of orbit of a point is given by letting the group of holonomies to act over the point. The concepts of $\omega$-limit sets and recurrence for discrete dynamical systems or flows have analogues on the theory of foliations and, as for the discrete case, they play an important role in understanding of the dynamical behavior of a foliation.

\begin{definition}
Let $\mathcal F$ be a $C^r$-foliation on a Riemannian manifold $M$ and denote by $\mathcal T$ a complete transversal. For $x\in \mathcal T$
\begin{itemize}
\item[1)] the orbit of $x$ is the set
\[\mathcal  O(x) = \mathcal F(x) \cap \mathcal T.\]
\item[2)] The $\omega$-limit set of a point $x\in \mathcal T$ is the relatively compact saturated subset
\[\omega(x) = \bigcap_{S\subset \mathcal O(x) , \#S <\infty} \overline{\mathcal O(x) - S} \subset \mathcal T. \]
\item[3)] The orbit $\mathcal O(x)$ is said to be recurrent if $\omega(x) \cap \mathcal O(x)  \ne \emptyset$. In this case we also say that $\mathcal F(x)$ is a recurrent leaf and that $x$ is a recurrent point.
\end{itemize}
\end{definition}
%

A leaf $L$ of a codimension-one foliation $\mathcal F$ is said to be resilient if it captures itself by an holonomy contraction. More precisely, a leaf $L$ of a codimension-one foliation $\mathcal F$ is said to be resilient if there exists a holonomy map $f$ defined in a neighborhood $U\subset T$ of a point $x\in L\cap T$ and another point $y\in L\cap U$ such that $y\ne x$,
\[f(x) = x, \quad \text{and} \quad f^n(y)\rightarrow x.\]

Entropy is one of the most important concepts in dynamical systems and has many variations such as metric entropy and topological entropy for example. These two, metric and topological entropy, has been very effective in the classification of chaotic behaviors of diffeomorphisms and flows, as well as in problems of determining the typical behavior of certain classes of dynamical systems. With the idea of creating a similar dynamical invariant for pseudogroups $C^1-$actions E. Ghys, R. Langevin and P. Walczak \cite{GLW} introduced the concept of geometric entropy for pseudogroup $C^1-$actions. Geometric entropy has been a very important dynamical invariant in the theory of foliation dynamics. This invariant captures, in some sense, the exponential complexity of the holonomies taking into account the transverse asymptotic behavior and the asymptotic growth of the leaf altogether. In what follows we do not give the explicit definition of geometric entropy since it is not needed along the paper but instead we need the following result relating the existence of resilient leaves with the positivity of the entropy.

\begin{theorem}[\cite{Hurder2010}] \label{Hurder}
If $\mathcal F$ is a codimension$-1$ $C^1-$foliation of a compact connected Riemannian manifold $(M,g)$, then $h_g(\mathcal F) > 0$ if, and only if, $\mathcal F$ has a resilient leaf, where $h_g(\mathcal F)$ denotes the geometric entropy of $\mathcal F$.
\end{theorem}

Theorem \ref{Hurder} was first proved for $C^2-$foliations by Ghys-Langevin-Walczak \cite{GLW} and later extended to $C^1$-foliations by S. Hurder \cite{Hurder2010} using techniques from ergodic theory and topological dynamics of flows. The reader may also find a discussion on the proof given by Hurder in \cite[Section $4.6$]{Walczak}.

\begin{remark}
Given $\mathcal F$ a $C^r-$foliation of a Riemannian manifold $(M,g)$, the geometric entropy $h_g(\mathcal F)$ depends not only on the foliation but also on the given Riemannian metric $g$. Although, if $(M,\mathcal F)$ is a compact $C^{\infty}-$foliated manifold it is true that if the geometric entropy $h_g(\mathcal F)$ vanishes and $g_1$ is another Riemannian metric on $M$ then $h_{g_1}(\mathcal F)$ also vanishes. That is, for compact $C^{\infty}-$foliated manifolds the positivity of the geometric entropy of a foliation $\mathcal F$ is independent of the Riemannian metric chosen on $M$ (see \cite[Corollary $13.3.3$]{CandelConlonI}).
\end{remark}

\section{Splitting charts} \label{sec:separating}
%
%

As remarked in the introduction, one of the main ideas of this paper is to work with foliations which have some recurrent leaves and take advantage of the recurrence property to ``cut and paste'' the leaves and produce a resilient leaf. In particular we want to perturb leaves which are not compact. In a broad setting this seems to be an unreasonable goal since there is no natural way to spread a perturbation on a small ball of the manifold to the whole foliation. Although, if we restrict our attention to foliations whose leaves can be cut open along a certain codimension$-1$ compact submanifold  (e.g, a leaf diffeomorphic to the strip $\mathbb R \times (0,1)$ or to the cylinder $\mathbb R \times S^1$ can be decomposed in two connected components by cutting it open in the direction of the second coordinate), then we can hope to do a local perturbation and spread it on the direction of the cutting. The definition of family of splitting charts meets this requirements and allows us to do $C^r-$local perturbations of certain foliations in a similar way that we do for diffeomorphisms.

In what follows, by a chain of foliated charts for a foliation $\mathcal F$ we mean a sequence $\phi_i: U_i \rightarrow [0,1]^m$, $1\leq i \leq n$, of foliated charts of $\mathcal F$ such that $U_i \cap U_{i+1}\ne \emptyset$ for every $i=0,1,...,n-1$.

\begin{definition}
Let $\Sigma$ be a $C^r$ surface, denote by $O_{\varepsilon}(X)$ the $\varepsilon-$neighborhood of a subset $X\subset \Sigma$. We say that a $C^r-$curve $\gamma:[0,1]\rightarrow \Sigma$ splits the surface $\Sigma$ if it is simple (i.e, it has no self intersection) and $\Sigma \setminus O_{\varepsilon}(\gamma([0,1]))$ is not connected, for any $\varepsilon >0$ sufficiently small. 
\end{definition}

\begin{definition} \label{definition:jointly.separating.charts}
Let $M$ be a $C^r$ manifold and $\mathcal F$ be a codimension$-q$, $C^r-$foliation of $M$ by surfaces (in particular $q=m-2$. We say that a finite family of foliated charts $\{\phi_i:U_i \rightarrow [0,1]^{2}\times [0,1]^q\}_{1\leq i \leq n}$ has the jointly splitting property, or that it is a family of jointly splitting charts, if there exists a compact surface with boundary $\Sigma$, endowed with a splitting curve $\gamma$, and a $C^r-$diffeomorphism 
\[ \mathfrak p: \overline{U_1\cup \ldots \cup U_n} \rightarrow \Sigma \times D^q,\]
where $D^q$ denotes the $q$-dimensional disk in $\mathbb R^q$, which maps the foliation of $\overline{U}$ induced by $\mathcal F$ to the foliation of $\Sigma \times D^q$ by the sets $\Sigma \times \{t\}$, $t\in D^q$, and such that for each $t\in D^q$ the curve $\mathfrak p^{-1}(\gamma \times \{t\})$ splits the leaf $L$ of $\mathcal F$ in which it is contained.

\end{definition}

\begin{remark}
Observe that for any foliation $\mathcal F$ by surfaces of a given manifold $M$ we can find a finite family $\{\phi_i:U_i \rightarrow [0,1]^{m}; 1\leq i \leq n\}$ of jointly splitting charts. Indeed, let $(\psi,U)$ be a foliated chart of a regular atlas of $\mathcal F$. Thus $\psi(U) \subset [0,1]^m$ is a rectangular neighborhood foliated by the sets $\psi(U)\cap ([0,1]^{2} \times \{c\})$, $c\in \mathbb R^{q}$, where $q=m-2$ is the codimension of $\mathcal F$. This latter foliation of $\psi(U)$ clearly admits a finite family of jointly splitting charts $\{\widetilde{\phi_i}:\widetilde{U_i} \rightarrow [0,1]^{m}\}_{1\leq i \leq n}$, $\widetilde{U_i}\subset \psi(U)$, and consequently taking the pull back by $\phi$ we obtain a family of jointly splitting charts for $\mathcal F$.
\end{remark}


\begin{example} \label{example1} 
%

Let $\phi_t$ be an irrational flow on $\mathbb T^2$. Let $M = \mathbb T^2 \times [0,1]$ and define $\mathcal F$ to be the foliation given by $\mathcal F(r,x) = \{\phi_t(x): t\in \mathbb R\} \times [0,1]$ for $0\leq r\leq 1$ and $x\in \mathbb T^2$. Observe that this foliation admits a finite family of jointly splitting charts and has zero geometric entropy for it is a Riemannian foliation (see \cite[Section $16$]{Hurder2009}) . Now, we consider $C$ the cylinder given by
\[C = [0,1]\times \left\{ (y,z) \in (0,1)^2: \left(y-\frac{1}{2}\right)^2+\left(z-\frac{1}{2}\right)^2 < \frac{1}{16} \right\}.\]
Let $\pi: [0,1]^3 \rightarrow \mathbb T^2 \times [0,1]$ be the natural projection and take $N:= M \setminus C$. In $N$ we take the foliation $\mathcal G$ given by $\mathcal G = \mathcal F \setminus C$.
The foliation $\mathcal G$ is a $C^{\infty}$ Riemannian minimal foliation such that every leaf has infinity genus.

\end{example}



\begin{example} Consider the representation $\rho : \pi_1(\mathbb T^2) \rightarrow \operatorname{Diff}_{+}^{\infty}(\mathbb S^1)$ given by $\rho(0)=Id$ and $\rho(1) = R_{\alpha}$ (here we are implicitly using the fact that $\pi_1(\mathbb T^2)  =\mathbb Z^2$) where $R_{\alpha}$ is an irrational rotation on $\mathbb S^1$. In $[0,1]^2 \times S^1$ let $\mathcal F_0$ be the foliation with leaves $\{[0,1]^2 \times \{\theta\}\}_{\theta \in S^1}$. Let $L_0 = [0,1]^2 \times \{\theta_0\}$ be an arbitrary leaf of $\mathcal F_0$ and take 
\begin{align*}
U_0^1&:= [1/3,1]\times [1/3,2/3]\times  [\theta_0-\delta, \theta_0+\delta]  \subset L_0, \\
U_0^2& :=([0,2/3] \cup [5/6,1])\times [1/3,2/3] \times [\theta_0-\delta, \theta_0+\delta] \subset L_0,\end{align*}
with $\delta$ small enough. Now, on the boundary of $[0,1]^2 \times S^1$ we define the following two identifications:
\begin{itemize}
\item[1)] $(x,y,\theta) \sim (x',y',\theta')$ if $x=0,x'=1, y'=y$ and $\theta' = \rho(0)(\theta) = \theta$;
\item[2)] $(x,y,\theta) \sim (x',y',\theta')$ if $x'=x, y=0, y'=1$ and $\theta' = \rho(1)(\theta) = R_{\alpha}(\theta)$.
\end{itemize}
Applying both identifications we obtain a smooth foliation $\mathcal F$ of $\mathbb T^3$, called the suspension of the representation $\rho$, whose leaves are cylinders (i.e, every leaf is homeomorphic to $\mathbb R \times S^1$). Denote by $U_1, U_2$ the sets obtained by applying the equivalente relations $(1)$ and $(2)$ to the sets $U_0^1$ and $U_0^2$  respectively. Then, $U_1$ and $U_2$ are domains of certain foliated charts $(\phi_1,U_1), (\phi_2,U_2)$ for $\mathcal F$ and furthermore $\overline{U_1 \cup U_2}$ is smoothly diffeomorphic to $\Sigma \times [0,1]$ where $\Sigma = \{(x,y)\in \mathbb R^2 : 1/2 \leq x^2+y^2 \leq 1\}$. Therefore, $\{(\phi_1,U_1), (\phi_2,U_2)\}$ is a family of jointly splitting charts.
\end{example}

\begin{example} Let $\mathcal F$ be the Reeb foliation on the solid tori $D^2 \times S^1$ and call $L$ the unique compact leaf of $\mathcal F$ (which is a two torus). In $L$ we can take an essential loop $\gamma$ such that the holonomy along $\gamma$ is trivial. It is easy to see that one may take a family of jointly splitting charts covering $\gamma$. \end{example}

\begin{example} Let $M$ be a three dimensional Riemannian manifold and $\mathcal F$ be a codimension-$1$ $C^r$-foliation of a compact Riemannian manifold $M$ by compact surfaces. By a classical result due to G. Hector \cite{Hector77} and to D. Epstein, K. Millet and D. Tischler \cite{EMT} the set $\mathcal L$ of leaves without holonomy is residual. Let $L$ be any leaf in $\mathcal L$ and let $\gamma$ be an essential loop on $L$. Let $\{\phi_i : U_i \rightarrow [0,1]^3\}_{1\leq i \leq k}$ be a chain of foliated charts covering $\gamma([0,1])$. Since $L$ is without holonomy, there exists an open set $U\supset \gamma([0,1])$, $U \subset U_1 \cup \ldots \cup U_k$, such that $\overline{U}$ is $C^r-$diffeomorphic to a product of the form $\Sigma \times [0,1]$ where $\Sigma = L \cap \overline{U}$. Thus, by restricting the foliated charts $\phi_i$ to $U_i \cap U$ we obtain a family of jointly splitting charts for $\mathcal F$. \end{example}

%
%
%
%

We close this section by defining water slide functions.

\begin{definition}\label{defi:waterslide}
Consider the unit square $[0,1]^m$ foliated by the submanifolds $[0,1]^{m-q} \times \{c\}$, $c\in [0,1]^q$. Let $P \subset [0,1]^m$ be a rectangular neighborhood  and $\pi_i(P) = [0,1]$ for $1\leq i \leq m-q$. A $C^r-$water slide function on $P$ is a $C^r$ diffeomorphism $\varphi:P \rightarrow P$ satisfying:
\begin{itemize}
\item[a)] $\xi$ fixes the first $p$-coordinates;
\item[b)] for any $c \in [0,1]^q$ there exists a unique $d \in [0,1]^q$ such that
\begin{equation} \label{item:a}
\varphi(\{1\} \times [0,1]^{p-1} \times \{c\}) = \{1\} \times [0,1]^{p-1} \times \{d\};
\end{equation}
\item[c)] $\varphi$ is equal to identity in a neighborhood of $(\{0\} \times [0,1]^{m-1}) \cap P$;
\item[d)] \label{condition:c} $D^{(\alpha)}(\varphi (x) - \operatorname{Id}) = 0$ for $|\alpha|\leq r$ and $x\in (\{1\}\times [0,1]^{m-1}) \cap P$.
\end{itemize}
\end{definition}


For the sake of simplicity,  when  ~\eqref{item:a} occurs we say that the water slide function $\varphi$ slides from the height $c$ to the height $d$. Given any two points $a, b\in [0,1]^m$, if ~\eqref{item:a} is true for $c = \pi^q(a)$ and $d=\pi^q(b)$ where $\pi^q :\mathbb R^m \rightarrow \mathbb R^q$ is the projection on the last $q$-coordinates, then we also say that the water slide function $\varphi$ slides from the $a$-height to the $b$-height.

\section{The $C^r$-perturbation procedure} \label{section:cutandglue}

By the definition of the $C^r-$Epstein topology one can make use of the first axiom of Epstein \cite{Epstein} to $C^r-$perturb a foliation on the space $\operatorname{Fol}^q_r(M)$, that is, if we take any $C^r$ diffeomorphism $j$ which is $C^r-$close to $\operatorname{Id}$ then $j(\mathcal F)$ is a $C^r-$foliation which is $C^r-$close to $\mathcal F$. However, this kind of perturbation does not fit our purposes since it essentially preserves all the characteristics of the leaves, in particular it preserves resilient (resp. non resilient) leaves. We now describe how that on a family of splitting charts we can use water slide functions to $C^r-$perturb the given foliation.

From now on let $\mathcal F$ be a codimension-$q$, $q=m-2$, $C^r-$foliation of an $m$-dimensional Riemannian manifold $M$, $1\leq r < \infty$ and let
\begin{itemize}
\item $\mathcal V \subset \operatorname{Fol}^q_r(M)$ be a neighborhood of $\mathcal F$;
\item $\eta= \{\eta_{i}: W_{i} \rightarrow [0,1]^{2}\times [0,1]^q\}_{1\leq i \leq n}$ a family of jointly splitting charts;
\item $W:= \bigcup_{i=1}^{n} W_i$.
\end{itemize}
In what follows we will use a $C^r-$water slide function $\xi:[0,1]^{2}\times [0,1]^{q} \rightarrow [0,1]^{2}\times [0,1]^q$, whose construction depends on $\mathcal V$, to construct a $C^r-$foliation $\mathcal G \subset \mathcal V$ whose charts restricted to $U$ are basically the restricted charts of $f$ perturbed by the water slide function.

In the next proposition we prove that the family of jointly splitting charts can always be reduced to a family of jointly splitting charts with only one or two foliated charts. Moreover, the construction of these charts is such that they are subfoliated by a codimension$-(q+1)$ $C^r-$foliation $\mathcal Q$, whose leaves are fully inside a given plaque chain, and a $q-$dimensional foliation $\tau$ whose leaves are transverse to the plaques of $\mathcal F$.

\begin{proposition} \label{prop:coordinates}
There exists a family of jointly splitting charts $\Phi:=\{\phi_i:U_i \rightarrow [0,1]^m\}_{1\leq i \leq n}$ with $U_i \subset W$,  $n\in \{1,2\}$, and $C^{r}-$foliations $\mathcal Q$ and $\mathcal \tau$ of $U = U_1 \cup U_2$ such that
\begin{itemize}
\item[1)] $\mathcal Q$ is a codimension$-(q+1)$ foliation of $U$ and, for each $u\in U$, $\mathcal Q(u) \subset \mathcal F(u)$;
\item[2)] $\mathcal \tau$ is a one-dimensional foliation of $U$ such that $\tau(u)$ is transverse to $\mathcal F(u)$ for every $u\in U$.
\end{itemize}   
\end{proposition}
\begin{proof}

As in Definition \ref{definition:jointly.separating.charts} let $\overline{W}$ be diffeomorphic to $\Sigma \times D^q$ and $\gamma$ to be an splitting curve of $\Sigma$.

\vspace{.2cm}

\noindent \textit{First case:} $\gamma(0) \ne \gamma(1)$.\\
Let $(\tilde{\phi}, U)$ be a foliated chart of $\mathcal F$ such that $\gamma([0,1])\subset U \subsetneq W$. The $C^r-$foliation $Q$ of $U$ given by the leaves $\{\tilde{\phi}^{-1}([0,1]^{p-1}\times \{c\}) \}_{c\in [0,1]^{q+1}}$ is a codimension$-(q+1)$ $C^r-$foliation satisfiying $(1)$. Finally, let $\tau$ be the foliation given by $\{\phi^{-1}(\{d\} \times [0,1]^{q})\}_{d\in [0,1]^p}$. As $\tau$ satisfies $(2)$ the proof is finished for this case.

\vspace{.2cm}

\noindent \textit{Second case:} $\gamma(0)=\gamma(1)$. \\
We can assume $\gamma(0)\in P_1$. Let $t_0\in (0,1]$ be such that $\gamma(t) \in P_1$ for every $t_0\leq t\leq 1$ and let 
$(\phi_1,U_1)$ be a foliated chart of $\mathcal F$ such that $\gamma([0,t_0]) \subset U_1 \subsetneq W$. Since we can take $U_1$ arbitrarily thin, by a change of coordinates if necessary, we can assume that $\phi_1^{-1}(\{0\} \times [0,1]^{m-p}), \tau_1:=\phi_1^{-1}(\{1\} \times [0,1]^{m-p}) \subset W_1$.
Consider the $C^r-$foliation $\widetilde{Q}$ of $U_1$ given by the leaves $\{\phi_1^{-1}([0,1]^{p-1}\times \{c\}) \}_{c\in [0,1]^{q+1}}$. 

Let $\tau_0:= \phi_1^{-1}(\{0\} \times [0,1]^{m-p})$ and $\tau_1:=\phi_1^{-1}(\{1\} \times [0,1]^{m-p})$. Observe that the foliation $\widetilde{Q}$ naturally induces a $C^r-$function $h: \tau_0 \rightarrow \tau_1$, $x \mapsto \widetilde{Q}(x)\cap \tau_1$. Now, we restrict our attention to the foliated chart $(\eta_1,W_1)$. In $W_1$ consider the $C^r-$flow $\widetilde{\varphi}$ given by
\[\widetilde{\varphi}_t(x) := \eta_1^{-1}( t\cdot \eta_1(h^{-1}(x))+ (1-t) \cdot \eta_1(x)). \]
Observe that $\widetilde{\varphi}_1(h(x))=x$ thus $\widetilde{\varphi}_1(\tau_1)=\tau_0$. Similarly, we can define on $\widetilde{W}$ a $C^r-$flow $\tilde{\sigma}$ given by
\[\widetilde{\sigma}_t(x) :=\phi_1^{-1}( (1-t)\cdot \phi_1(x)+ t\cdot \phi_1(h(x)). \]
For $\widetilde{\sigma}$ we have $\widetilde{\sigma}_1(\tau_0)=\tau_1$. Now, by taking smooth bumping functions $\rho_0,\rho_1$ on neighborhoods of $\tau_0$ and $\tau_1$ respectively, with $\rho_0 | \tau_0 = Id$ and $\rho_1| \tau_1 = Id$, we can take a $C^r-$flow $\varphi$ defined on $\widetilde{W} \cup \widetilde{\varphi}([0,1]\times \tau_1)$ such that $\varphi$ is equal to $\widetilde{\varphi}$ out of $\widetilde{W}$ and is equal to $\widetilde{\sigma}$ in $\widetilde{W}$. Finally, we define $\mathcal Q$ as the foliation whose leaves are the sets $\widetilde{Q}(x)\cup \widetilde{\varphi}([0,1]\times (\widetilde{Q}(x) \cap \tau_1))$. The second foliated chart $(\phi_2,U_2)$ is obtained by taking $U_2:= W_2 \cap (U_1 \cup \varphi([0,1]\times \tau_1))$.

Now the foliation $\tau$ is defined in a similar way. We define $\tau(x):= \phi^{-1}(\{d\} \times [0,1]^{q})$ for each $x\in \widetilde{W}$  and for $x\in P_n \setminus \widetilde{W}$ we define $\tau(x) = \varphi_t(\tau_0)$, $0\leq t \leq 1$.

\end{proof}

From now on we will work with the family of jointly splitting charts given by Proposition \ref{prop:coordinates}.

Let $V_i$ denote the set of plaques of $U_i$, $1\leq i \leq 2$. Observe that by the construction made in the proof of Proposition \ref{prop:coordinates} we have
\[h_{U_1U_{2}} ( V_1) = V_{2}, \quad \text{ and } \quad  h_{U_2U_{1}} ( V_2) = V_{1}.\]
Let $V:=V_1 \cup V_2$.  Assume that $\mathcal V$ is given by the family of foliated charts $\{\varphi_i : X_i \rightarrow \mathbb [0,1]^m\}_{i\in \Lambda}$, a locally finite family of compact sets $\{K_{\lambda} \subset X_{\lambda}\}_{\lambda \in \Lambda}$ satisfying properties $(1)-(3)$ from Proposition ~\ref{proposition:epstein} and a set of real numbers $\{\delta_{i}\}_{i\in \Lambda}$. Since $\{K_j\}_{j\in \Lambda}$ is locally finite, for each $p\in \overline{U}$ we can take an open set $K(p) \ni p$ such that $K(p)$ intersects at most a finite number of sets $K_j$. Since $\overline{U}$ is compact, we can cover $\overline{U}$ with a finite number of sets $K(p)$, $p\in \overline{U}$, and consequently only a finite number of sets $K_j$ intersect $\overline{U}$. Let $K_{\lambda_1},\ldots,K_{\lambda_l}$ be such sets. For simplicity we will sometimes denote $K_{\lambda_j}$ by $K_j$, $1\leq i \leq l$. 
Let $\widetilde{K_j}$, $1\leq j \leq l$, the smallest compact set such that $K_j \subset \widetilde{K_j}$ and $\widetilde{K_j}$ cross completely the set $\overline{U}$, that is, for $i\in \{1,2\}$ the set $\phi_i(\overline{V}\cap \widetilde{K_j}) \subset [0,1]^m$, $1\leq j \leq l$, is a rectangular neighborhood with $\pi_1(\phi_i(\overline{V}\cap \widetilde{K_j})) = [0,1]$. Consequently, the intersection of $\phi_i(\overline{V}\cap \widetilde{K_j})$ with the sides $\{0\}^p \times [0,1]^{q}$ and $\{1\}^p \times [0,1]^{q}$ are non-degenerate rectangular neighborhoods.


\vspace{.2cm}

\subsection{The subordinated partition} \label{subsec:subordinated}
For each $1\leq j \leq l$, denote by $\kappa_j$ the ``plaque saturation'' of $K_j \cap V$ inside $V$, that is, \[\kappa_j:= \bigcup_{i,j\in \{1,2\}} h_{U_iU_j}(\{Q: Q \text{ is a plaque of  } V_i \text{ with } Q\cap K_j \ne \emptyset\}).\]
Now consider $\mathcal P$ the joining of the partitions $\{\kappa_j ; V\setminus \kappa_j\}$
\[\mathcal P := \bigwedge_{j=1}^l \{\kappa_j; V\setminus \kappa_j\},\]
that is, $\mathcal P$ is the partition of $V$ by sets of the form $B_1 \cap \ldots \cap B_l$ where, for each $1\leq j \leq l$, $B_j = \kappa_j$ or $B_j = V \setminus \kappa_j$. Observe that we have constructed a finite partition $\mathcal P$ of $V$ with the following property:
\begin{itemize}
\item for each $P\in \mathcal P$, $\overline{\phi_i(P\cap V_i)}$ is a rectangular neighborhood of the form \[\overline{\phi_i(P\cap V_i)}=[0,1]^{p} \times \left( \prod_{i=q}^m[a_i,b_i] \right);\]
\item for $P,Q \in \mathcal P$ with $P\ne Q$, either $\overline{\phi_i(\overline{P}\cap \overline{Q})} = \emptyset$ or $\overline{\phi_i(\overline{P}\cap \overline{Q})} = [0,1]^{p} \times \{A\},$ for some $A\in [0,1]^q$, and is in the boundary of $\phi_i(P\cap Q)$.
\end{itemize}

Let $\mathfrak p: \overline{U} \rightarrow \Sigma \times D^q $ be a $C^r-$diffeomorphism and let $\varepsilon>0$ be small enough so that $\Sigma \setminus O_{\varepsilon}(\gamma([0,1])$ is not a connected set. Consider the foliated charts $(\phi_1,U_1), (\phi_2,U_2)$ constructed in Proposition \ref{prop:coordinates} and denote by $\overline{\phi}_i:\overline{U_i} \rightarrow [0,1]^m$, $i=1,2$, the extension of $\phi_i$ to the boundary of $U_i$. We know that $S_j:= \overline{\phi_1}^{-1}(\{j\}^p\times [0,1]^q) \cup \overline{\phi_2}^{-1}(\{j\}^p\times [0,1]^q)$ is a connected set for $j=0,1$. Thus, given any plaque chain $P_1,P_2$ in $U_1\cup U_2$ we can consider $O_{P_1,P_2}^1$ and $O_{P_1,P_2}^2$ the connected components of  $(P_1\cup P_2) \setminus \mathfrak p^{-1}(O_{\varepsilon}(\gamma([0,1]) \times \{d\})$, where $d= \pi_2 \circ \mathfrak p (x)$ for any $x\in P_1 \cup P_2$, which contain $S_1 \cap (P_1\cup P_2)$ and $S_2 \cap (P_1\cup P_2)$ respectively. Let 
\[O^j:= \bigcup O_{P_1,P_2}^j, \quad j=0,1\]
where the union is over all the plaque chains $P_1,P_2$ in $U_1\cup U_2$.
Observe that by eventually decreasing the set of numbers $\delta_i$ we may assume that for any $1\leq i \leq l$, $K_i$ never intersect $O^1$ and $O^2$ simultaneously.

\subsection{The perturbation} \label{subsection:perturbation}

We proceed to the construction of the perturbation procedure. The first step towards the construction of the perturbation is to perturb the plaques inside $V_1$. After this we will extend the perturbed plaques to perturb plaques on $V_2$, and then to each $K_i$ intersecting $\overline{U}$.

\vspace{.5cm}

\noindent {\underline{First step:}} Perturbing the plaques of $V_1$.

Given the partition $\mathcal P = \{P_1,...,P_s\}$ of $V$ as above. Let $\xi_1,\ldots, \xi_s$ be $C^r-$water slide functions defined on $\phi_1(P_1\cap V_1),...,\phi_1(P_s\cap V_1)$ respectively. In particular, by items $(a)$ and $(b)$ of Definition \ref{defi:waterslide} and by the construction of the foliations $\mathcal Q$ and $\tau$ (see Proposition ~\ref{prop:coordinates}) we have that 
\begin{itemize}
\item $\xi_i(\phi_1(\mathcal Q(x) \cap P_i\cap U_1))$ is fully inside a plaque of $V_i$, for any $x\in V_i$;
\item $\xi(\phi_1(\tau(x))) \subset \tau(x)$, for every $x\in P_i \cap U_1$,
\end{itemize}
for every $1\leq i \leq s$.
 Let $\xi:[0,1]^m \rightarrow [0,1]^m$ be defined by $\xi(x) = \xi_{I(x)}(x)$ where $x\in P_{I(x)}\cap V_1$. Then $\xi$ is a $C^r-$water slide function and $\xi$ can be taken arbitrarily $C^r-$close to identity by taking the $\xi_i$'s arbitrarily $C^r-$close to identity. Now define $\psi_1: V_1 \rightarrow V_1$ by
\[\psi_1 = \phi_1^{-1} \circ \xi \circ \phi_1.\]

\noindent {\underline{Second step:}} Extending the perturbation to the plaques of $V_2$.

Given any plaque $P_2$ of $V_2$, there exists one, and only one, plaque $P_1$ on $V_1$ such that $P_1\cap P_2 \ne \emptyset$. Now, for $p\in P_2$ we take 
\[\psi_2(p):=\mathcal F(\psi_1(\mathcal Q(p) \cap P_1)) \cap \tau(p).\]
Since the foliations $\mathcal Q$ and $\tau$ are $C^r$-foliations of $U$, this defines a $C^r-$function $\psi_2:V_2 \rightarrow V_2$. Also, since $\psi_2  | V_1 \cap V_2 = \psi_1 | V_1 \cap V_2$ we have defined a $C^r-$function $\psi: U \rightarrow U$ such that $\psi  | V_1 := \psi_1$ and $\psi|V_2:=\psi_2$. Also observe that $\psi$ can be taken arbitrarily $C^r-$close to identity by taking $\xi$ sufficiently close to identity.

\vspace{.5cm}

\noindent {\underline{Third step:}} Extending $\psi$ to the family of sets $\{K_i\}_{i\in \Lambda}$.

As we have remarked on the last paragraph of Subsection \ref{subsec:subordinated}, we can assume that the $K_i$'s have the property that each of them cannot simultaneously intersect $O^1$ and $O^2$. Let $1\leq i \leq l$ be fixed.
\begin{itemize}
\item[i)] If $K_i \cap O^1 \ne \emptyset$ we define $\psi_{K_i}:K_i \rightarrow K_i$ as 
\[\psi_{K_i}(p)=p \text{ for } p\in K_i \setminus U \text{ and } \psi_{K_i}(p)= \psi(p) \text{ for } p\in U\cap K_i.\]
\item[ii)] Assume $K_i \cap O^2 \ne \emptyset$. For $p\in K_i \setminus U$ we define $\psi_{K_i}(p)=p$. If $p\in U \cap K_i$ 
consider $\{Q_1(p),Q_2(p)\}$ be the chain of plaques on $U$, with $Q_j(p) \in U_j$, $j=1,2$, such that the water slide function $\xi$ slides from the $\phi_1(Q_1(p))$-height to the $\phi_1(\mathcal Q(p)\cap U_1)$-height. 
Then we define
\[\psi_{K_i}(p) = \psi(\eta(p)),\]
where $\eta: K_i\cap U \rightarrow K_i$ is the function $p \mapsto \tau(p) \cap (Q_1(p) \cup Q_2(p))$. 
\item[iii)] If $K_i \cap U = \emptyset$ we define $\psi_{K_i}=Id$
\item[iv)] If $K_i \subset U$ and $K_i \cap (O^1 \cup O^2) = \emptyset$ then we define $\psi_{K_i}=\psi | K_i$.
\end{itemize}
If $\lambda \in \Lambda$ is such that $\lambda \notin \{\lambda_1,\ldots,\lambda_l\}$ then we define $\psi_{K_{\lambda}}=Id$. By the construction observe that, for each $\lambda \in \Lambda$, $\psi_{K_{\lambda}}$ is arbitrarily $C^r-$close to identity if $\psi$ is chosen sufficiently close to identity. Therefore, if $\xi$ is sufficiently $C^r-$close to identity then each $\psi_{K_{\lambda}}$ is also as $C^r-$close to identity as we want.
 
The last step consists in defining the perturbed codimension-$q$ foliation $\mathcal G$ which is $C^r-$close to the original foliation $\mathcal F$. As it is well known, this can be done by defining a coherent system of local charts, that is, an atlas of local charts $\Omega$ with the property that whenever $P$ and $Q$ are plaques in distinct charts of $\Omega$ then $P\cap Q$ is open in $P$ and open in $Q$. Equivalently, a foliated atlas $\Omega = \{(U_{\alpha},\varphi_{\alpha}): \alpha \in \Lambda\}$ of a codimension-$q$ foliation $\mathcal F$ is coherent if, for any $U_{\alpha} \cap U_{\beta} \ne \emptyset$, the change of coordinates $g_{\alpha \beta} := \varphi_{\alpha} \circ \varphi_{\beta}^{-1} : \varphi_{\beta}(U_{\alpha} \cap U_{\beta}) \rightarrow \varphi_{\alpha}(U_{\alpha} \cap U_{\beta})$ is of the form
\[g_{\alpha \beta }(x,y) = (g_1(x,y), g_2(y))\]
where $x\in \mathbb R^{m-q}$, $y\in \mathbb R^q$, $(x,y) \in \varphi_{\beta}(U_{\alpha} \cap U_{\beta})$ (see \cite[Section $1.2$]{CandelConlonI} for a proof of this equivalence).
%

\vspace{.2cm}

\noindent {\underline{Fourth step:}} 
We finally define a new system of local charts $\omega_{\lambda} : \operatorname{int}(K_{\lambda})\rightarrow (0,1)^m$, $\lambda \in \Lambda$, by taking $\omega_{\lambda} := \varphi_{\lambda} \circ \psi_{K_{\lambda}}^{-1}$, which is a $C^{r}-$function. Observe that the atlas $ \Omega:= \{ (\operatorname{int}(K_{\lambda}),\omega_{\lambda}): \lambda \in \Lambda\}$ is coherent. Indeed, let $\omega_{i},\omega_{j} \in \Omega$, $i,j\in \Lambda$, with $\operatorname{Int}(K_{i}) \cap \operatorname{Int}(K_{j}) \ne \emptyset$.
\begin{itemize}
\item If $\operatorname{Int}(K_{i}) \cap \operatorname{Int}(K_j) \subset U^c$, then $\omega_i \circ \omega_j^{-1} | \operatorname{Int}(K_i) \cap \operatorname{Int}(K_j) = Id$.
%
%

\item If $\operatorname{int}(K_i), \operatorname{int}(K_j) \subset U$ then $\psi_{K_i} = \psi| K_i, \psi_{K_j} = \psi | K_j$ and we have \[\omega_i \circ \omega_j^{-1} | \operatorname{Int}(K_i) \cap \operatorname{Int}(K_j) = \varphi_i \circ \varphi_j^{-1} |\operatorname{Int}(K_i) \cap \operatorname{Int}(K_j).\]

\item If $K_i \cap O^{v} \ne \emptyset$ and $K_i \cap O^v \ne \emptyset$, for $v=1$ or $v=2$, then $\omega_i\circ \omega_j^{-1} | \operatorname{Int}(K_i) \cap \operatorname{Int}(K_j)= \varphi_i \circ \varphi_j^{-1} | \operatorname{Int}(K_i) \cap \operatorname{Int}(K_j)$.

\item If one of the sets $\operatorname{int}(K_i), \operatorname{int}(K_j)$ is inside $U$ and the other is not contained in $U$. We can assume $\operatorname{int}(K_i) \subset U$ and $\operatorname{int}(K_j) \not \subset U$. In this case $\psi_{K_i} = \psi| K_i$ and either $\psi_{K_j} | (K_j \cap U)= \psi | (K_j \cap U)$ if $K_j\cap O^1 \ne \emptyset$ or $\psi_{K_j} |(K_j \cap U) = \psi \circ \eta | (K_j \cap U)$ if $K_j \cap O^2 \ne \emptyset$. If $K_j \cap O^1 \ne \emptyset$ then 
\[\omega_i \circ \omega_j^{-1} | \operatorname{Int}(K_i) \cap \operatorname{Int}(K_j) = \varphi_i \circ \varphi_j^{-1} |\operatorname{Int}(K_i) \cap \operatorname{Int}(K_j).\]
If $K_j \cap O^2 \ne \emptyset$ then
\[\omega_i \circ \omega_j^{-1} | \operatorname{Int}(K_i) \cap \operatorname{Int}(K_j) = \varphi_i \circ \eta \circ \varphi_j^{-1} |\operatorname{Int}(K_i) \cap \operatorname{Int}(K_j).\]

\item If $K_i \cap O^1 \ne \emptyset$, $K_j \cap O^2 \ne \emptyset$ (the case $K_j \cap O^1 \ne \emptyset$, $K_i \cap O^2 \ne \emptyset$ uses the same argument) we have \[\operatorname{Int}(K_i) \cap \operatorname{Int}(K_j) \subset U \setminus (O^1 \cup O^2).\]
By the definitions we have $\omega_i | \operatorname{Int}(K_i) \cap \operatorname{Int}(K_j) = \varphi_i \circ  \psi^{-1} | \operatorname{Int}(K_i) \cap \operatorname{Int}(K_j)$ and $\omega_j | \operatorname{Int}(K_i) \cap \operatorname{Int}(K_j) = \varphi_j \circ \psi_{K_j}^{-1} = \varphi_j \circ \eta^{-1} \circ \psi^{-1}$. Thus
\[\omega_i \circ \omega_j^{-1} | \operatorname{Int}(K_i) \cap \operatorname{Int}(K_j) = \varphi_j \circ \eta \circ \varphi_j^{-1} | \operatorname{Int}(K_i) \cap \operatorname{Int}(K_j) .\]
\end{itemize}

In the first three cases it is trivial that $\omega_i \circ \omega_j^{-1}$ is of the form $(x,y) \mapsto (h_1(x,y), h_2(y))$, $(x,y) \in (0,1)^p\times (0,1)^q$ since $\varphi_i \circ \varphi_j^{-1}$ is of such form. On the fourth and fifth itens observe that the function $\eta$ maps plaques plaques of $K_j$ to plaques of $K_j$ in the sense that, if $p,q\in K_j \cap U$ belong to the same plaque then $\eta(p), \eta(q) \in K_j \cap U$ also belong to the same plaque. In particular $\varphi_j \circ \eta^{-1} \circ \varphi_j^{-1}$ is also of the form $(x,y) \mapsto (h_1(x,y), h_2(y))$, $(x,y) \in (0,1)^p\times (0,1)^q$, which implies that 
\[\omega_i \circ \omega_j^{-1} | \operatorname{Int}(K_i) \cap \operatorname{Int}(K_j) = [\varphi_j \circ \varphi_j^{-1}] \circ [\varphi_j \circ \eta \circ \varphi_j^{-1} ]| \operatorname{Int}(K_i) \cap \operatorname{Int}(K_j) \]
has the desired form.

\begin{lemma}
The $C^r-$foliation $\mathcal G$ given by the coherent $C^r-$atlas $\Omega=\{(\operatorname{int}(K_i), \omega_i) : i\in \Lambda\}$ is in $\mathcal V$.
\end{lemma}
\begin{proof}
Indeed, for $i\in \Lambda$, $i\notin \{\lambda_1,\ldots,\lambda_l\}$, there is nothing to do since $\omega_i = \varphi_i$. Take $i\in \{\lambda_1,\ldots ,\lambda_l\}$ and take $g_i:\operatorname{Int}(U_i) \rightarrow \mathbb R^q$ a transversal projection of $\mathcal G$, that is, $g_i = \rho_2 \circ \omega_i$. Thus for any multi-index $\alpha$ with $|\alpha| \leq r$ we have
\[|D^{(\alpha)}(\rho_2 \circ \omega_i \circ \varphi_i^{-1}-\rho_2)| = |D^{(\alpha)}(\rho_2 \circ \varphi_i \circ \psi^{-1}_{K_i} \circ \varphi_i^{-1}-\rho_2)|.\]
As remarked along the construction of the functions $\psi_{K_i}$ we can guarantee that $\psi_{K_i}$ is $C^r-$close to identity by taking $\xi$ suficiently $C^r-$close to identity. Therefore, we can take $\xi$ sufficiently $C^r-$close to identity so that 
\begin{align*}
 |D^{(\alpha)}(\varphi_i \circ \psi^{-1}_{K_i} \circ \varphi_i^{-1}-Id)| & < \min\{\delta_{\lambda_1},\ldots, \delta_{\lambda_l}\} \\
 \Rightarrow |D^{(\alpha)}(\rho_2 \circ \omega_i \circ \varphi_i^{-1}-\rho_2)| & <\delta_{i}, \text{ for } i\in \{\lambda_1,\ldots , \lambda_l\},
 \end{align*}
 which implies $\mathcal G \in \mathcal V$ as we wanted to show.
\end{proof}

\begin{remark}
We remark that the $C^r-$perturbation procedure just described does not apply for $r=\infty$, that is, if $r=\infty$ the procedure just described will not necessarily generate a $C^{\infty}-$foliation $\mathcal G$ which is $C^{\infty}-$close to the previously given foliation. Indeed, if $\xi$ is $C^{\infty}-$close to identity then the derivatives of all orders are uniformly bounded and, consequently, $\xi$ is an analytic function. But then, condition $(c)$ from Definition \ref{condition:c} would imply that 
$\xi$ is actually the identity function.
\end{remark}


\section{Creating a resilient leaf} \label{sec:creatingresilient}
In order to prove Theorem ~\ref{theorem:main} we will need the following simple auxiliary Lemma.

\begin{lemma} \label{lemma:auxiliary}
Let $I=[0,1]$ and $\overline{\xi}: I \rightarrow I$ be a $C^r-$function which is $C^r-$close to identity and such that $\overline{\xi}(t_0) = t_0$ for a certain $t_0 \in (0,1)$. Then we can define a $C^r-$water slide function $\xi: [0,1]^{m} = [0,1]^{m-1}\times I \rightarrow  [0,1]^{m}$ which is $C^r-$close to identity and such that
\begin{itemize}
\item[i)] $\pi \circ \xi= \overline{\xi}$, where $\pi:[0,1]^m \rightarrow I$ is the projection on the last coordinate;
\item[ii)] $\pi_j \circ \xi = Id$ where $1\leq j \leq m-1$;
\item[ii)] given $c\in [0,1]^{m-1}$, $a\in I$, $\xi$ slides from the $(c,a)$-height to the $(c,\overline{\xi}(a))$-height;
\item[iii)] $\xi | ([0,1]^{m-1}\times\{t_0\}) = \operatorname{Id}$.
\end{itemize}
\end{lemma}
\begin{proof}
Let $L:[-2,2] \rightarrow [-2,2]$ be a smooth function satisfying
\begin{itemize}
\item for $t$ in a neighborhood of $0$ we have $L(t) = 1$;
\item there exists $t_1 < 2$ such that $L(t)=0$ for every $t_1 \leq t \leq 2$.
\end{itemize}
Let $\xi_0:[0,1]^2 \rightarrow [0,1]^2$ be defined by 
\[\xi_0(t,x) = (t, L(t)x+(1-L(t))\overline{\xi}(x)).\] 
Since $L$ is a smooth function, the $j$-th derivative of $L$ is uniformly bounded. Then, by taking $\overline{\xi}$ close enough to $Id$ on the $C^r-$topology we have that $\xi_0$ is $C^r-$close to $Id$. Note that  $\xi_0(t,t_0) = (t,L(t)t_0+(1-L(t))\overline{\xi}(t_0)) = (t,t_0)$ for every $t\in [0,1]$. Finally, the desired water slide function is given by $\xi: [0,1]^{m} = [0,1]^{m-2} \times [0,1] \times I \rightarrow  [0,1]^{m}$
\[\xi ( c, t, x) = (c, \xi_0(t,x)).\]
As $\xi_0$ is as close as we want to $Id$ on the $C^r-$topology, then $\xi$ can be taken to be arbitrarily $C^r-$close to $Id$.
\end{proof}

\begin{proof}[Proof of Theorem ~\ref{theorem:main}]

Let $\mathcal V$ be any $C^r-$neighborhood of $\mathcal F$. Let $\Phi= \{\phi_i:B_i \rightarrow \mathbb R^m\}_{i\in \Lambda}$ be a family of jointly splitting charts. 
By Proposition \ref{prop:coordinates} we can assume that this system has only two charts $(\phi_1,U_1)$ and $(\phi_2,U_2)$. We can assume without loss of generality that $\gamma([0,1]) \cap U_1 \ne \emptyset$. Now, as done in Section~\ref{section:cutandglue}, call $\mathcal P = \{ P_1,...,P_s\}$ the partition of $U_1 \cup U_2$ obtained by taking the intersections of $U_1 \cup U_2$ with the locally finite family of compact sets $\{K_i\}$ describing by $\mathcal V$.

Take $a\in \gamma([0,1]) \cap U_1$. Again there is no loss of generality in assuming that $a=\gamma(0)$. By increasing the compact sets $\{K_i\}$ if necessary we can assume that for a certain $1\leq w \leq s$ we have $a\in \operatorname{Int}(P_w)$. Let $\mathcal T_a$ be a transversal passing through $a$, that is, $a\in \mathcal T_a:= \phi_1^{-1}([0,1]^{m-1} \times \{\mathfrak c\}) \subset U_1$ for some $\mathfrak c\in [0,1]$. By the definition of the partition $\mathcal P$ we know that $\overline{\phi_1(P_w \cap U_1)}$ is a rectangular neighborhood in $[0,1]^m$, say
\[\overline{\phi_1(P_w\cap U_1)} = \prod_{i=1}^m [\alpha_i, \beta_i], \quad 1\leq \alpha_i<\beta_i \leq 1, \text{ for } 1\leq i \leq m.\]
Let $\mathfrak R: \overline{\phi_1(P_w \cap U_1)} \rightarrow [0,1]^m$ be given by $\mathfrak R (x_1,...,x_m) = \left( \frac{\alpha_1-x_1}{\alpha_1-\beta_1}, \ldots,  \frac{\alpha_m-x_m}{\alpha_m-\beta_m}\right)$ and denote $\mathfrak j: \overline{\phi_1(\mathcal T_a \cap P_w)} \rightarrow [0,1]$ the function $\mathfrak j:= \pi_m \circ \mathfrak R$. On $\overline{\phi_1(\mathcal T_a \cap P_w)}$ we take $\rho$ the distance obtained by pulling back through $\mathfrak j$ the standard distance on $[0,1]$. Given two points $c,d \in \overline{\phi_1(\mathcal T_a \cap P_w)}$ we will say that $c$ is higher $d$ if
\[\rho(c,\phi_1(a)) = \rho(c,d)+\rho(d,\phi_1(a)),\]
and we say that a point $x\in \mathcal T_a \cap P_w$ is between two points $z,w \in \mathcal \mathcal T_a \cap P_w$ if the point $\mathfrak j(\phi_1(x))$ is inside the interval with extremes $\mathfrak j(\phi_1(z))$ and $\mathfrak j(\phi_1(w))$. To simplify the notation we will also denote $\max (x,y) := \max (\mathfrak j (\phi_1(x)), \mathfrak j (\phi_1(y)))$, $x,y\in \mathcal T_a \cap P_w$.

Let $f$ be the holonomy along the loop $\gamma$, that is, $f=h_{\gamma} : D \rightarrow \mathcal T$ where $D\subset \mathcal T$ is an open set on the complete transversal $\mathcal T$. The holonomy $f$ naturally induces a function $f_a:D(a) \rightarrow \mathcal T_a$, from an open set $D_a \subset \mathcal T_a$ with $\gamma(0)=a\in D_a$ to $\mathcal T_a$, by taking
\[f([x]) =[f_a(x)] , \quad x\in \mathcal T_a,\]
where $[x] \in \mathcal T$ denotes the element in $\mathcal T$ given by the plaque of $x$. More precisely, to define $[x]$ we are implicitly using the fact that a complete transversal $\mathcal T$ is a disjoint union of spaces of plaques $T_U:=U /(\mathcal F | U)$ of the foliated charts of $\mathcal F$ (where $U$ varies among the domains of the foliated charts). Therefore a point in $\mathcal T$ is identified uniquely with a plaque of the atlas (see \cite[Definition $1.3.4$]{Walczak}).
%
%

Denote $J=\mathcal T_a \cap \operatorname{Int}(P_w)$. We construct a sequence of points $(b_n)_n$ in $J$ as follows. Take $b_1 \ne a$ be a point in $J$ and close enough to $a$ so that $a_1:=f(b_1) \in J$. Now, by the continuity of $f$ we can choose $b_2$ between $a$ and $\max\{b_1,a_1\}$ such that $a_2:=f(b_2)$ is still between $a_1$ and $a$. Having chosen $b_1,b_2,\ldots, b_{n-1}$, again by continuity of $f$ we can pick $b_n$ between $a$ and $\max \{ b_{n-1},a_{n-1}\}$ such that $a_n:=f(b_n)$ is still between $a_{n-1}$ and $a$.


Denote, $a':=\mathfrak j (\phi_1(a))$, $b_n':=\mathfrak j (\phi_1(b_n))$, $a_n':=\mathfrak j (\phi_1(a_n))$.
 Let $g:=\max(a_1',b_1')$ and take $q_2 \in [0,1]$ such that $<q_2$. Consider the function $\overline{\xi}_0: [0,1] \rightarrow [0,1]$ given by:
\begin{itemize}
\item for each $i$ we define $\overline{\xi}_0(b_i'):= a_i'$;
\item $\overline{\xi}_0(x) = x$ for every $x\in [0,1]$ with $x\leq a'$ or $x\geq q_2$;
\item $\overline{\xi}_0$ is linear on the intervals $[g,q_2]$ and $[b_i',b_{i+1}']$, $i\geq 1$.
\end{itemize}
Applying Lemma ~\ref{lemma:auxiliary} to $\overline{\xi}$, we can define a $C^r-$water slide function $\tilde{\xi}:[0,1]^m \rightarrow [0,1]^m$ such that $\tilde{\xi}$ slides from the $(c,b_i')$-height to the $(c,a_i')$-height and $\tilde{\xi}([0,1]^{m-1}\times \{a'\})=\operatorname{Id}$. Finally, we define a $C^r-$water slide function $\xi : \phi_1(P_w \cap U_1) \rightarrow \phi_1(P_w\cap U_1)$ by taking $\xi:= \mathcal R^{-1} \circ \tilde{\xi} \circ \mathcal R$.

Furthermore, by taking $b_1$ close enough to $a$ we have that $\overline{\xi}$ is $C^r-$close to identity and consequently $\xi$ is as $C^r-$close to identity as we want. Let $\xi_j : \phi_1(P_j \cap U_1) \rightarrow \phi_1(P_j \cap U_1)$ to be the identity function for all $j1\leq j \leq s, j\ne w$ and define
\[\xi(x):= \xi_{I(x)}(x),\]
where $x\in P_{I(x)} \cap U_1$.

%
%

Applying the $C^r-$perturbation procedure described in Section ~\ref{section:cutandglue}, using the function $\xi$ constructed above, we obtain a $C^r-$foliation $\mathcal G \in \mathcal V$. Let $g$ be the holonomy induced by $f$ after the perturbation procedure. More precisely, by item (i) from the statement of the Theorem \ref{theorem:main}, in a sufficiently small neighborhood $\mathcal U'\subset \mathcal T$ of $a$ we can take the holonomy $g : \mathcal U'\subset \mathcal T \rightarrow \mathcal T$ given by 
\[g([x])= [\psi \circ f_a(x)], \quad \text{ for all } x\in D_a,\]
where $\psi:U \rightarrow U$ is defined following the construction of Subsection \ref{subsection:perturbation} and $U' = \{[x] : x\in \mathcal T_a\}$.
Then, 
\[g([b_n]) = [\psi \circ f_a (b_n)] = [\phi^{-1} \circ \xi( \phi (a_n))] =  [\phi^{-1} (\phi(b_{n+1}))] = [b_{n+1}], \quad \forall  n.\]
In particular $b_1 \ne a$ and $g^n([b_1]) \rightarrow [a]$, showing that the leaf of $\mathcal G$ which contains $a$ is resilient. 
\end{proof}

\begin{proof}[Proof of Corollary ~\ref{coro}]
Apply Theorem \ref{theorem:main} to the foliation given in Example \ref{example1}.
\end{proof}

\section{Final considerations and problems}

We say that a foliation $\mathcal F \in \operatorname{Fol}^q_r(M)$, where $(M,g)$ is a $C^r-$Riemannian manifold, has $C^r-$robustly zero geometric entropy if there exists a neighborhood $\mathcal V \subset \operatorname{Fol}^q_r(M)$ of $\mathcal F$ in the $C^r-$Epstein topology such that for every $\mathcal G \in \mathcal V$ we have $h_g(\mathcal G)=0$.

If we assume $M$ to be a compact smooth Riemannian manifold, our results shows that under the hypothesis of Theorem ~\ref{theorem:main} one may use recurrence to approximate a given foliation $\mathcal F$ by foliations with positive geometric entropy. Then, Theorem ~\ref{theorem:main} provides an obstruction for a foliation to have $C^r-$robustly zero geometric entropy. A natural question that rises in this context, and which was previously addressed by S. Hurder in the $C^1$ context \cite[Problem $9.11$]{HurderSurvey}), is:

\begin{problem}
Give examples (if there are any) of $C^r-$foliations $\mathcal F \in  \operatorname{Fol}^q_r(M)$, $r\geq 1$, of a compact Riemannian manifold $M$ which have $C^r-$robustly zero geometric entropy.
\end{problem}

As remarked in the introduction, it was proved by S. Crovisier \cite{Crovisier} that the set of $C^1-$diffeomorphisms which have robustly zero topological entropy is exactly the closure of the set of Morse-Smale diffeomorphisms. Can we obtain a similar classification for foliations?

\begin{problem}
Classify the set of all foliations $\mathcal F \in \operatorname{Fol}^q_1(M)$ with $C^1-$stably zero entropy where $M$ is a compact Riemannian manifold.
\end{problem}

\section*{Acknowledgement}
The author expresses his hearty gratitude to the anonymous referee who helped to improve considerably the redaction of the paper. The author also thanks R\'egis Var\~ao for useful conversations on the subject. The author had the financial support of FAPESP process \# 2016/05384-0 and FAEPEX process \#2334/16.

\bibliographystyle{plain}
\bibliography{Referencias}

\end{document}